\documentclass[1 [leqno,11pt]{amsart}
\usepackage{amssymb, amsmath}
\usepackage{bbm}
\def\refer#1{~\ref{#1}}
\def\refeq#1{~(\ref{#1})}
\def\ccite#1{~\cite{#1}}

\def\longformule#1#2{
\displaylines{ \qquad{#1} \hfill\cr \hfill {#2} \qquad\cr } }
\def\inte#1{
\displaystyle\mathop{#1\kern0pt}^\circ }

\let\pa=\partial

\let\d=\delta
\let\e=\varepsilon

\let\lam=\lambda

\let\f=\frac

\let\p=\psi
\let\om=\omega

\let\D=\Delta

\let\Om=\Omega
\let\wt=\widetilde
\let\wh=\widehat

\def\cF{{\mathcal F}}
\def\cG{{\mathcal G}}
\def\cH{{\mathcal H}}
\def\cI{{\mathcal I}}

\def\cN{{\mathcal N}}

\def\grad{\nabla}

\def\virgp{\raise 2pt\hbox{,}}
\def\cdotpv{\raise 2pt\hbox{;}}

\def\eqdefa{\buildrel\hbox{\footnotesize def}\over =}

\def\C{\mathop{\mathbb C\kern 0pt}\nolimits}
\def\DD{\mathop{\mathbb D\kern 0pt}\nolimits}
\def\EE{\mathop{{\mathbb E \kern 0pt}}\nolimits}
\def\K{\mathop{\mathbb K\kern 0pt}\nolimits}
\def\N{\mathop{\mathbb N\kern 0pt}\nolimits}
\def\Q{\mathop{\mathbb Q\kern 0pt}\nolimits}
\def\R{\mathop{\mathbb R\kern 0pt}\nolimits}
\def\SS{\mathop{\mathbb S\kern 0pt}\nolimits}
\def\ZZ{\mathop{\mathbb Z\kern 0pt}\nolimits}
\def\TT{\mathop{\mathbb T\kern 0pt}\nolimits}
\def\P{\mathop{\mathbb P\kern 0pt}\nolimits}
\newcommand{\ds}{\displaystyle}

\def\dv{\mbox{\rm div}}

\def\dive{\mathop{\rm div}\nolimits}

\def\na{\nabla}
\def\p{\partial}
\def\uh{u^{\rm h}}

\def\xh{x_{\rm h}}
\def\Rh{ {\R^2_{\rm h}} }
\def\Rv{ {\R_{\rm v}}}
\def\presspO{p_{\rm h}}
\def\vapp{\wt u^{\rm h}_{\e, \rm app}}
\def\uapp{ u^{\rm h}_{\e, \rm app}}
\def\xih{\xi_{\rm h} }

\def\LVLH {L^\infty_{\rm v} (L^2_{\rm h})}
\def\1{\mathbbm{1}}

\newcommand{\w}[1]{\langle {#1} \rangle}
\newcommand{\beq}{\begin{equation}}
\newcommand{\eeq}{\end{equation}}
\newcommand{\ben}{\begin{eqnarray}}
\newcommand{\een}{\end{eqnarray}}
\newcommand{\beno}{\begin{eqnarray*}}
\newcommand{\eeno}{\end{eqnarray*}}
\newcommand{\andf}{\quad\hbox{and}\quad}
\newcommand{\with}{\quad\hbox{with}\quad}

\newtheorem{thm}{Theorem}[section]
\newtheorem{lem}{Lemma}[section]
\newtheorem{rmk}{Remark}[section]
\newtheorem{col}{Corollary}[section]

\begin{document}
\title[Global solutions of 3-D
 Navier-Stokes system]
{Remarks on the global solutions of 3-D
 Navier-Stokes system with one slow variable}
 \author[J.-Y. CHEMIN]{Jean-Yves Chemin}
\address [J.-Y. Chemin]%
{Laboratoire J.-L. Lions, UMR 7598 \\
Universit\'e Pierre et Marie Curie, 75230 Paris Cedex 05, FRANCE }
\email{chemin@ann.jussieu.fr}
\author[P. ZHANG]{Ping Zhang}%
\address[P. Zhang]
 {Academy of
Mathematics $\&$ Systems Science and  Hua Loo-Keng Key Laboratory of
Mathematics, The Chinese Academy of Sciences, Beijing 100190, CHINA}
\email{zp@amss.ac.cn}

\date{03/17/2014}

\begin{abstract} By applying  Wiegner' method in \cite{Wiegner}, we first prove the large time decay estimate
for the global solutions of a 2.5 dimensional Navier-Stokes system,
which is a sort of singular perturbed 2-D Navier-Stokes system in
three space dimension. As an application of this decay estimate, we
give a simplified proof for the global wellposedness result in
\cite{cg3} for 3-D Navier-Stokes system with one slow variable. Let us also
mention that compared with the assumptions for the initial data in
\cite{cg3}, here the assumptions in Theorem \ref{slowvarsimplifie}
are weaker.
\end{abstract}

\maketitle


\vskip 0.2cm


\setcounter{equation}{0}
\section{Introduction}

The first part of this paper is devoted to the study of the
following system
\begin{equation*}
 {\rm (NS 2.5D)}\qquad\left\{\begin{array}{l}
\displaystyle \pa_t u_\e^{\rm h} + \dv_{\rm h} (u^{\rm h}_\e\otimes
u^{\rm h}_\e) -\D_\e u^{\rm h}_\e =-\grad_{\rm h} p^{\rm h}_{\e},
\qquad (t,x)\in\R^+\times\R^3,
\\
\displaystyle \dv_{\rm h}\,u^{\rm h}_\e  = 0, \\
\displaystyle { u^{\rm h}_\e}|_{t=0}=u_0^{\rm h}
\end{array}\right.
\end{equation*}
where~$u^h(t,x_{\rm h}, z) = \bigl( u^1 (t,\xh,z),u^2(t,\xh,z)\bigr)
$ with~$(\xh,z)$ in~$\R^2_{\rm h} \times \Rv,$ $\na_{\rm
h}=(\p_1,\p_2),$ $\D_{\rm h}=\p_1^2+\p_2^2$ and~$\D_\e=\D_{\rm
h}+\e^2\p_{3}^2$. Moreover,~$\e$ is a (small) positive parameter.
Let us first point out that in the case when~$\e$ is equal to~$0$,
the above system is simply the two dimensional  incompressible
Navier-Stokes system with an initial data depending on the real
parameter~$z$.

The motivation of studying this system~(NS2.5D) comes from  the study of global wellposedness of the three dimensionnal incompressible Navier-Stokes system~(NS3D) which is
\begin{equation*}
 {\rm (NS3D)}\qquad\left\{\begin{array}{l}
\displaystyle \pa_t u + \dv (u\otimes u) -\D u=-\grad p,\quad
\mbox{in} \:
\R^+ \times \R^3\\
\displaystyle \dv u = 0, \\
\displaystyle { u}|_{t=0}=u_0
\end{array}\right.
\end{equation*}
where the initial data are said to be "slowly varying" with respect to the vertical variable
which means that they are of the form
\begin{equation*}
u_0(x)=(u_0^{\rm h}(\xh,\e x_3),0)  \with \dive_h u^{\rm h}_0
(\cdot,z) =0.
\end{equation*}

The interest of this type of initial data is that they are relevant
tools to investigate the problem of global wellposedness for~(NS3D).
First of all, they provide a class of large initial data for the
system~(NS3D) which are globally wellposed and which do not have
symmetries. Indeed, the following result holds (see\ccite{cg3},
Theorem~1 and Proposition~1.1).
\begin{thm}
\label{theoquasi2D}
{\sl Let~$v_0^{\rm h}=(v_0^1,v_0^2)$ be a  horizontal, smooth
 divergence free  vector field on~$\R^3$ (i.e.~$v_0^{\rm h}$ is  in~$L^2(\R^3)$ as well as all its derivatives),
 belonging, as well as all its derivatives, to~$L^2(\R_{z}; \dot H^{-1}(\Rh))$. Then, there exists a positive~$\e_0$ such that,  if~$\e\leq \e_0$, the initial data
$$
u_{0}^{\e} (x) = \bigl(v_0^{\rm h}(x_{\rm h},\e x_3),0\bigr)
$$
generates a unique, global   solution~$u^{\e}$ of~(NS3D).

Moreover, if~$v_0^{\rm h} (\xh,z) \eqdefa \wt v_0(\xh)g(z)$ then if $\e$ is small enough,
$$
\|u_0^\e\|_{\dot B^{-1}_{\infty,\infty}(\R^3)} \geq \frac 1 4 \|\wt
v_0^{\rm h}\|_{\dot B^{-1}_{\infty,\infty}(\Rh)}
\|g\|_{L^\infty(\Rv)} \quad\hbox{\rm where}\quad \ds\|a\|_{\dot
B^{-1}_{\infty,\infty} (\R^d)} \eqdefa \sup_{t>0} t^{\frac 12}
\|e^{t\Delta} a\|_{L^\infty}.
$$
}
\end{thm}

This last inequality ensures that the above global wellposedness
result is not a consequence of the Koch and Tataru theorem
(see\ccite{kochtataru}) which claims that if  the regular initial
data~$u_0$ of~(NS3D) is sufficiently small in the norm of the
space~${\rm BMO}^{-1}(\R^3),$ then
it generates a global smooth solution. Here, let us simply recall
that the space~${\rm BMO}^{-1}(\R^3)$ is continuously imbedded  in
the Besov space~$\dot B^{-1}_{\infty,\infty}(\R^3)$.

Moreover, such slowly varying initial data allows to say something
about the geometry  of the set~$\cG$ of initial data in~$\dot
H^{\frac12}(\R^3)$ which generates global solution  in the
space~$C(\R^+;\dot H^{\frac 12}(\R^3))$. In\ccite{gipcras}, I.
Gallagher, D. Iftimie and F. Planchon  proved
 that this set is open and connected (see also\ccite{adt} and\ccite{gip} for the same property in more sophisticated spaces).
  Using slowly varying perturbations,  I. Gallagher and the two authors proved in\ccite{cgz} that through
   any point of~$\cG$, there are uncountable   lines  of arbitrary length in the space~$\dot B^{-1}_{\infty,\infty}(\R^3),$
and thus in the Sobolev space~$\dot H^{\frac 12}(\R^3),$ which is
continuously imbedded in the space~$\dot
B^{-1}_{\infty,\infty}(\R^3)$. An interpretation in terms of support
of the Fourier transform of the initial data is presented
in\ccite{cgm}.

Such initial data appears also in the study of the problem
concerning the openness to the set~$\cG$  for weak topology
(see\ccite{bg} and\ccite{bcg2}).

  \medbreak
 The way Theorem\refer{theoquasi2D} is proved in\ccite{cg3}  is as follows.  Let us consider~$\uh(t,\xh,z)$ the (global) solution of the 2D incompressible Navier-Stokes system
 $$
{\rm (NS2D_{3})} \left\{
\begin{array}{c}
\partial_t \uh  + \uh  \cdot \nabla_h \uh  -\Delta_h  \uh  = -\nabla_h p^{\rm h} \quad \mbox{in} \:
\R^+ \times \R^2\\
\dive_h \uh  = 0\\
\uh |_{t=0} = \uh_0(\cdot ,z).
\end{array}
\right.
$$
This system is globally wellposed for any~$z$ in~$ \R$,  and the
solution is smooth in (two dimensional) space, and in time. Let us
define the approximate solution \ben \label{definvapp}
\wt{u}^\e_{\rm app}(t,x) = ( \uh(t,x_h,\e x_3) , 0 )\andf
  \wt{p}^\e_{\rm app}(t,x)  =   p^{\rm h}  (t,x_h,\e x_3).
  \een
  and let us search the solution of~(NS3D) as
$$
u^{\e} =  \wt{u}^\e_{\rm app} +  R^{\e} .
$$
 Classical computations leads to
\beq\label{approximateb}
\begin{split}
&\partial_t R^\e + R^\e \cdot \nabla R^\e -\Delta
R^\e+\wt{u}^\e_{\rm app} \cdot \nabla R^\e+R^\e \cdot \nabla
\wt{u}^\e_{\rm app} =
 \wt F^\e-\nabla q^\e \with\\
 &\qquad\qquad\qquad\qquad\qquad F^\e \eqdefa \bigl( \e^2\partial^2_z \uh, \e \partial_z  p^{\rm h}\bigr)  (t,x_h,\e x_3).
\end{split}
\eeq It is  easy to observe that, if we have good uniform estimates
on~$\wt{u}^\e_{\rm app}$ and that  ${\wt F_\e}$  tends to~$0$  when
$\e$ tends to~$0$ in a space like~$L^2(\R^+;\dot H^{-\frac
12}(\R^3))$, then the global wellposedness is proved (see\ccite{cg3}
for the details). Here, the fact that~$\e^2(\partial_z^2 \uh )
(t,\xh,\e x_3)$ appears as an error term is conceptually not
satisfactory because it is a term coming from the viscosity and thus it
is supposed to produce decay or regularity and yet here it is a
source of some technical difficulty.

\medbreak The idea here is to substitute ~${\rm (NS2D_{3})}$
by~(NS2.5D). We have to prove global wellposedness for ~(NS 2.5D)
with regular initial data, which is not difficult, and also the
space time estimate in~$L^p$, which should   of course be
independent of the parameter~$\e$. This is the new point of this
paper. The precise statement is the following.
\begin{thm}
\label{estimfondNS2.5} {\sl Let~$\uh_0$ and~$\nabla_{\rm h}\uh_0$ be
in~$L^2(\R^3)\cap L^\infty(\Rv;L^2(\Rh))$. Then~$\uh_0$ generates a
unique global solution to~(NS2.5D) in the space~$L^\infty_{\rm loc}
(L^4(\R^3))$.

Moreover, if in addition~$\uh_0$ belongs to~$L^\infty(\Rv;\dot
H^{-\delta} (\Rh))$ for some ~$\d$ in~$]0,1[$, then we have \ben
\nonumber
&\ds\int_0^{\infty} \|\nabla_{\rm h } \uh (t)\|_{L^\infty_{\rm v}(L^2(\Rh))}^2\,dt \leq  A_\d (\uh_0) \with\\
\label{estiomglobalfond}&A_\d (\uh_0) \eqdefa C_\d
\biggl(\frac { \|\nabla _{\rm h}\uh_0\|_{\LVLH}^2 \|\uh_0\|_{L^\infty_{\rm v} (\dot B^{-\d}_{2,\infty}(\Rh))}^{\frac 2 \delta}} {\|\uh_0\|_{\LVLH}^{\frac 2 \d}} + \|\uh_0\|_{\LVLH}^{2}\biggr)\\
\nonumber
& \qquad\qquad\qquad\qquad\qquad\qquad\qquad\qquad\qquad\qquad{}\times
\exp\big(C_\d\|\uh_0\|_{\LVLH}^2(1+\|\uh_0\|_{\LVLH}^2)\bigr)\,.
\een
}
\end{thm}

\begin{rmk} Following the procedure in Section \ref{sect3} and under the assumptions of Theorem
\ref{estimfondNS2.5}, we can prove more precise large time decay
estimates for $\uh(t)$ as follows \beno
\begin{split}
& \|\uh(t)\|_{L^\infty_{\rm v}(L^2_{\rm h})}^2\leq
Ce^{CC_0^2}\w{t}^{-\d}\with C_0\eqdefa \|u_0^{\rm
h}\|_{L^\infty_{\rm v}(L^2_{\rm h}\cap \dot{H}^{-\d}_{\rm
h})}\bigl(1+\|u_0^{\rm h}\|_{L^\infty_{\rm v}(L^2_{\rm h})}\bigr)\andf \\
&\|\na_{\rm h}\uh(t)\|_{L^\infty_{\rm v}(L^2_{\rm h})}^2\leq
CC_1\w{t}^{-(1+\d)}\with C_1\eqdefa \bigl(1+\|\na_{\rm h}u_0^{\rm
h}\|_{L^2}\bigr)e^{CC_0^2}.  \end{split} \eeno For a concise
presentation, we shall not present the details here.
\end{rmk}

The idea of the proof of this theorem is first  to perform energy
estimate in~$L^2(\R^3)$ which is very basic and far from being
enough here. Then we perform energy estimate in the horizontal
variables only for the vector field~$\uh$  and  its
vorticity~$\om^{\rm h}\eqdefa \partial_1u^2-\partial_2u^1$. The
maximum principle for the heat equation provides global
wellposedness of (NS2.5D). This is the purpose of the second
section.

Unfortunately, this global wellposedness results does not yield any
uniform bound for the solution with respect to~$\e$ in any ~$L^p$
norm for time. Thus we have no uniform global stability of such
global solutions in the sense that we want global stability  of the
global solutions of ~(NS2.5D) with the size of
 perturbation being independent of the parameter
$\e$.

In the third section, we introduce Wiegner's method in the context
of~(NS2.5D). This method has been introduced by M. Wiegner
in\ccite{Wiegner} in order to prove the large time decay estimate on
the~$L^2$ norm of a solution to the incompressible Navier-Stokes
system in the whole space.  For some developments and variations
about this questions of decay of the~$L^2$ norm in the whole space,
see the works\ccite{brandolese},\ccite{Schonbek}
and\ccite{Schonbek2}.

In the forth section, we apply Theorem\refer{estimfondNS2.5} to prove the following result.
\begin{thm}
\label{slowvarsimplifie} {\sl Let~$\uh_0$ be in~$H^1(\R^3)\cap
L^\infty(\Rv;\dot H^{-\delta} (\Rh))\cap L^\infty(\Rv;H^1(\Rh))$ for
some $\d\in]0,1[.$ Let us assume also~$\uh_0$ and~$\partial_z \uh_0$
belongs to~$L^2(\Rv;\dot H^{-\frac 12}\cap\dot H^{\frac 12} (\Rh))$.
Then for~$\e\leq \e_0$ depending only on the above norms, the
initial data
$$
u_{0,\e} (\xh, x_3) =\bigl( \uh_0(\xh, \e x_3),0\bigr)
$$
generates a global solution to~(NS3D) in the space~$C_b(\R^+;\dot
H^{\frac 12} (\R^3))\cap L^2(\R^+;\dot H^{\frac 32}(\R^3))$. }
\end{thm}
The idea of the proof of this theorem is to search a solution
of~(NS3D) as
$$
u^{\e} =  u^{\e}_{\rm app} +  R^{\e} \with u^\e_{\rm app}
(t,\xh,x_3)\eqdefa (\uh_\e(t,\xh,\e x_3),0)
$$
where~$(\uh_\e,p^{\rm h}_\e)$ is the solution of~(NS2.5D) with
initial data $\uh_0.$
 Classical computations leads to
\beq\label{approximatea}
\begin{split}
&\partial_t R^\e + R^\e \cdot \nabla R^\e -\Delta R^\e+u^\e_{\rm
app} \cdot \nabla R^\e+R^\e \cdot \nabla u^\e_{\rm app} =
 F^\e-\nabla q^\e \with\\
 &\qquad\qquad\qquad\qquad\qquad F^\e \eqdefa \bigl( 0, \e \partial_z  p^{\rm h}_\e\bigr)  (t,x_h,\e x_3).
\end{split}
\eeq The external force~$F^\e$ in \eqref{approximatea} is much
easier to be dealt with than the external force~$\wt F^\e$ in
\eqref{approximateb}.

\setcounter{equation}{0}
\section{Global wellposedness of~(NS2.5D) and maximum principle}

Let us first observe that the general theory of parabolic system
implies that, for any positive~$\e,$ a unique maximal
solution~$\uh_\e$ to~(NS2.5D) exists
in~$C([0,T^\star_\e[;H^1(\R^3))$ and that \beq \label{blowupNStype}
\mbox{if}\ \  T^\star_\e<\infty \ \Longrightarrow \  \forall p>3\,,\
\lim_{t\rightarrow T^\star_\e} \|\uh_\e\|_{L^p(\R^3)} =\infty. \eeq
 All the forthcoming computations will be valid for~$t$ less than~$T^\star_\e$. For simplicity, we shall
  drop  out the subscript $\e$ in what follows.

 Multiplying~(NS2.5D) by $\uh$ and then integrating the resulting
equation over $\Rh$ gives \beq
 \label{1.3c}
\f12\Bigl(\f{d}{dt}\|\uh(t,\cdot,z)\|_{L^2_{\rm
h}}^2-\e^2\p_3^2\|\uh(t,\cdot,z)\|_{L^2_{\rm h}}^2\Bigr)+\|\na_\e
\uh(t,\cdot,z)\|_{L^2_{\rm h}}^2=0, \eeq where $\na_\e=(\na_{\rm
h},\e\p_3)$. By integration \eqref{1.3c} for both time and  the
vertical variable~$z$, we get \beq \label{1.4a}
\f12\|\uh(t)\|_{L^2(\R^3)}^2+\int_0^t\|\na_\e
\uh(t')\|_{L^2(\R^3)}^2=\f12\|u_0^{\rm h}\|_{L^2(\R^3)}^2. \eeq
Moreover, from\refeq{1.3c}, we infer
$$
\f{d}{dt}\|\uh(t,\cdot,z)\|_{L^2_{\rm
h}}^2-\e^2\p_3^2\|\uh(t,\cdot,z)\|_{L^2_{\rm h}}^2\leq 0.
$$
The fact  that the heat flow is a contraction  in $L^p$ space implies that
\beq
\label{1.4}
\forall p\in [2,\infty]^2\,,\  \|\uh(t)\|_{L^p_{\rm v}(L^2_{\rm h})}\leq \|u_0^{\rm h}\|_{L^p_{\rm v}(L^2_{\rm h})}.
 \eeq
 There is no evidence that Equality\refeq{1.3c} provides an estimate of
 $$
 \sup_z \int_0^t \|\na_h\uh(t,\cdot,z)\|_{L^2_h}^2 dt'
 $$
 which is independent of~$\e$ for small~$\e$. Of course it is the case when~$\e=0$. This shows that the system (NS2.5D)
  is really a singular perturbation problem.

Because the nonlinear term in (NS2.5D) is a two dimensional one, we
have the following well-known equation on the vorticity~$\om^{\rm
h}=\p_1u^2-\p_2u^1$: \beq \label{1.2} \p_t\om^{\rm
h}+\uh\cdot\na_{\rm h}\om^{\rm h}-\D_\e\om^{\rm h}=0. \eeq Arguing
in the same way as the  above, we get \beq \label{1.3a}
\f12\Bigl(\f{d}{dt}\|\om^{\rm h}(t,\cdot,z)\|_{L^2_{\rm
h}}^2-\e^2\p_3^2\|\om^{\rm h}(t,\cdot,z)\|_{L^2_{\rm
h}}^2\Bigr)+\|\na_{\e}\om^{\rm h}(t,\cdot,z)\|_{L^2_{\rm h}}^2=0
\eeq which leads us to \beq \label{1.3b} \forall p\in
[2,\infty]^2\,,\  \
 \|\om^{\rm h}(t)\|_{L^p_{\rm v}(L^2_{\rm h})}\leq \|\om_0^{\rm h}\|_{L^p_{\rm v}(L^2_{\rm h})}
\eeq Now let us see that Inequalities\refeq{1.4} and \refeq{1.3b}
prevent the solution of (NS2.5D) from blowing up. Indeed,  by
interpolation between these two inequalities, we obtain
$$
\forall p\in ]2,\infty[\,,\  \|\uh(t)\|_{L^p(\R^3)} \leq  C_p \leq
\|u_0^{\rm h}\|^{\frac 2 p} _{L^p_{\rm v}(L^2_{\rm h})}\|\om_0^{\rm
h}\|^{1- \frac 2 p}_{L^p_{\rm v}(L^2_{\rm h})}.
$$
Assertion\refeq{blowupNStype} ensures that~$T^\star_\e$ is infinite.

\setcounter{equation}{0}
\section{Singular perturbation of M. Wiegner' method }\label{sect3}

Let us first recall Wiegner' method in \cite{Wiegner}. It consits in
truncating the frequency space with an appropriate time dependent
function. Given a positive function $g$ on~$\R^+,$  we define, for
a~$L^2$ function $a$ on~$\Rh$
$$
a_{\flat, g} (t) \eqdefa \cF_{\Rh}^{-1} \bigl(\1 _{S(t)} (\xi_{\rm h}) \wh a(\xi_{\rm h})\bigr)
\with S(t)\eqdefa \{\xi_{\rm h} \in \Rh\,/\ Ê|\xi_{\rm h} |\leq g(t)\}.
$$
The key lemma here is the following.
\begin{lem}
\label{Wiegnersingperturb} {\sl Let~$U$ be a regular function
on~$\R^+\times \Rh\times \Rv$ such that \beq
\label{Wiegnersingperturbeq1}
\partial_t \|U(t,\cdot,z)\|_{L^2_{\rm h}}^2-\e^2\partial_z^2 \|U(t,\cdot,z)\|_{L^2_{\rm h}}^2 +
2  \|\nabla_{\rm h} U(t,\cdot,z)\|_{L^2_{\rm h}}^2 \leq 0.
\eeq
Then for any positive function~$g$ on~$\R^+$, we have
$$
\longformule{ \|U(t)\|_{L^\infty_{\rm v} (L^2_{\rm h})}^2 \exp\Bigl
(2\int_0^t g^2(t')dt'\Bigr) \leq  \|U(0)\|_{L^\infty_{\rm v}
(L^2_{\rm h})}^2 } { {}+C \int_0^t  \|U_{\flat,g}
(t')\|_{L^\infty_{\rm v} (L^2_{\rm h})}^2  g^2(t') \exp\Bigl
(2\int_0^{t'} g^2(t'')dt''\Bigr) dt'\,. }
$$
}
\end{lem}
\begin{proof}
Let us write that
\beno
 \|\nabla_{\rm h} U(t,\cdot,z)\|_{L^2_{\rm h}}^2 & = &
 (2\pi)^{-2} \int_{\Rh} |\xi_{\rm h} |^2 |\wh U(t,\xi_{\rm h},z)|^2d\xih\\
 & \geq &
 (2\pi)^{-2} \int_{\{\xih\in \Rh\,/\ |\xih|\geq g(t)\}} |\xi_{\rm h} |^2 |\wh U(t,\xi_{\rm h},z)|^2d\xih\\
 &\geq &
 (2\pi)^{-2}g^2(t)  \int_{\Rh}  |\wh U(t,\xi_{\rm h},z)|^2d\xih\\
 &&\qquad\qquad\qquad\qquad\qquad
 -(2\pi)^{-2}g^2(t)  \int_{ \Rh} |\wh U_{\flat,g}(t,\xi_{\rm h},z)|^2d\xih\,.
\eeno Plugging this inequality in
Hypothesis\refeq{Wiegnersingperturbeq1} gives
$$
\partial_t \|U(t,\cdot,z)\|_{L^2_{\rm h}}^2-\e^2\partial_z^2 \|U(t,\cdot,z)\|_{L^2_{\rm h}}^2 +
2  g^2(t) \|U(t,\cdot,z)\|_{L^2_{\rm h}}^2 \leq 2g^2(t)  \|U_{\flat,g}(t,\cdot,z)\|_{L^2_{\rm h}}^2.
$$
The multiplication by~$\ds \exp\Bigl( 2\int_0^t g^2(t')dt'\Bigr)$
gives rise to
$$
\longformule{
\partial_t \biggl( \|U(t,\cdot,z)\|_{L^2_{\rm h}}^2 \exp\Bigl( 2\int_0^t g^2(t')dt'\Bigr)\biggr)
-\e^2\partial_z^2 \biggl( \|U(t,\cdot,z)\|_{L^2_{\rm h}}^2
\exp\Bigl( 2\int_0^t g^2(t')dt'\Bigr)\biggr) } { {}\leq 2
\|U_{\flat,g}(t)\|_{L^\infty_{\rm v}(L^2_{\rm h})}^2\, g^2(t)
\exp\Bigl( 2\int_0^t g^2(t')dt'\Bigr). }
$$
The maximum principle implies the lemma.
\end{proof}

In order to apply this lemma  with~$U=u^{\rm h}$ and $U=\om^{\rm
h}$, we need some control about low frequency part of~$u^{\rm h}$
and ~$\om^{\rm h}$.
\begin{lem}
\label{nonlinearbf} {\sl If~$\uh$ is a regular solution of~(NS2.5D),
then we have, for any positive function~$g$, \beno \|\uh_{\flat,
g}(t)\|_{L^\infty_{\rm v} (L^2_{\rm h})} &\leq & \|e^{t\Delta_{\rm
h}} \uh_0\|_{\LVLH}
+ C g^2(t) \int_0^t \|\uh(t')\|^2_{\LVLH} \, dt' \andf\\
\|\om_{\flat, g}^{\rm h}(t)\|_{L^\infty_{\rm v} (L^2_{\rm h})}  &
\leq & g(t)\|e^{t\Delta_{\rm h}} \uh_0\|_{\LVLH} + C g^2(t) \int_0^t
\|\uh(t')\|_{\LVLH}\|\om^{\rm h}(t')\|_{L^\infty_{\rm v} (L^2_{\rm
h})} \, dt'. \eeno }
\end{lem}
\begin{proof}
It is in fact a lemma about the heat equation with vanishing
diffusion in one direction. Let us consider~$a$ and~$f$   such
that~$\partial_t a -\D_\e a  =f .$ By definition of~$\Delta_\e$,
Duhamel's formula writes, after a Fourier transform with respect to
the horizontal variables,
$$
\longformule{ \wh a_{\flat, g} (t,\xih,z) = \1
_{S(t)}(\xih)e^{-t|\xih|^2} \frac 1 {(4\pi\e^2 t)^{\frac 12}}
\int_{\Rv} e^{-\frac {|z-z'|^2} {4\e^2t}} \wh a(0,\xih, z')dz' } {
{}+ \int_0^t \int_\Rv e^{-(t-t')|\xih|^2} \frac 1 {(4\pi\e^2
(t-t'))^{\frac 12} }
 e^{-\frac {|z-z'|^2} {4\e^2(t-t')}}  \1 _{S(t)}(\xih)\wh f (t',\xih,z')\,dz'dt'.
 }
$$
As the norm of an integral is less than or equal to the integral of the norm, we get, for any~$(t,z)$ in~$ \R^+\times\Rv$,
$$
\longformule{
\|a_{\flat, g} (t,\cdot,z)\|_{L^2_{\rm h}} \leq  \frac 1 {(4\pi\e^2 t)^{\frac 12}} \int_{\Rv} e^{-\frac {|z-z'|^2} {4\e^2t}} \bigl\|\bigl(e^{t\Delta_{\rm h}} a(0,\cdot, z')\bigr)_{\flat ,g}\bigr\|_{L^2_{\rm h}}\,  dz'
}
{
{}+
\int_0^t \int_\Rv  \frac 1 {(4\pi\e^2 (t-t'))^{\frac 12} }
 e^{-\frac {|z-z'|^2} {4\e^2(t-t')}}  \|\1 _{S(t)}(D_{\rm h}) f (t',\cdot,z')\|_{L^2_{\rm h}}\,
 dz'dt',
 }
$$ where $\1 _{S(t)}(D_{\rm h})$ denotes the Fourier multiplier with
$$
\1 _{S(t)}(D_{\rm h})g(x_{\rm h})=\cF^{-1}\bigl(\1
_{S(t)}(\xih)\widehat{g}(\xih)\bigr)(x_{\rm h}).
$$
Taking the $L^\infty$ norm with respect to the variable~$z$ gives \beq
\label{nonlinearbfdemoeq1} \|a_{\flat, g} (t)\|_{\LVLH} \leq
\bigl\|\bigl(e^{t\Delta_{\rm h}} a(0)\bigr)_{\flat
,g}\bigr\|_{\LVLH} + \int_0^t   \|\1 _{S(t)}(D_{\rm h}) f
(t')\|_{\LVLH} \,dt'. \eeq Using  Berstein inequality  in the
horizontal variables gives \beno
\| \bigl(e^{t\Delta_{\rm h}} \om_0^{\rm h}\bigr)_{\flat ,g}\|_{\LVLH}  & \lesssim &  g(t)  \|e^{t\Delta_{\rm h}} \uh_0\|_{\LVLH} \\
\|\1 _{S(t)}(D) \dive_{\rm h}( \uh(t')\otimes \uh(t')) \|_{\LVLH} &
\lesssim  &
g^2(t)  \|\uh(t')\|^2_{\LVLH}\andf\\
\bigl\|\1 _{S(t)}(D) \dive_{\rm h}(\om^{\rm h}(t') \uh(t'))
\bigr\|_{\LVLH} & \lesssim  & g^2(t)  \|\uh(t')\|_{\LVLH} \|\om^{\rm
h}(t')\|_{\LVLH}\,. \eeno This together with the fact that \beno
\uh(t)=e^{t\D_\e}\uh_0-\int_0^te^{(t-t')\D_\e}{\Bbb P}^{\rm
h}\dive_{\rm h}(\uh\otimes\uh)(t')\,dt' \eeno and \eqref{1.2}
implies the required inequalities, where ${\Bbb P}^{\rm
h}=I-\na_{\rm h}\D_{\rm h}^{-1}\dive_{\rm h}$ denotes the Leray
projection operator in two space dimension.
\end{proof}
Lemmas\refer{Wiegnersingperturb} and\refer{nonlinearbf}  can be summarized in the following corollary.
\begin{col}
\label{Wiegnersingperturbcorol} {\sl If~$\uh$ is a regular solution
of (NS2.5D), then we have \beno &&\|\uh(t)\|_{L^\infty_{\rm v}
(L^2_{\rm h})}^2 \exp\Bigl (2\int_0^t g^2(t')dt'\Bigr)
\leq  \|\uh_0\|_{L^\infty_{\rm v} (L^2_{\rm h})}^2+C \int_0^t  \biggl( \|e^{t'\Delta_{\rm h}} \uh_0\|^2_{\LVLH} \\
&&\ \qquad\qquad\qquad{}+ g^4(t')\Bigl( \int_0^{t'}
\|\uh(t'')\|^2_{\LVLH} dt''\Bigr)^2\biggr)
g^2(t') \exp\Bigl (2\int_0^{t'}g^2(t'')dt''\Bigr) dt'\andf\\
&&\|\om^{\rm h}(t)\|_{L^\infty_{\rm v} (L^2_{\rm h})}^2 \exp\Bigl
(2\int_0^t g^2(t')dt'\Bigr) \leq \|\nabla_{\rm h}
\uh_0\|_{L^\infty_{\rm v} (L^2_{\rm h})}^2
 +C \int_0^t  \biggl(g^2(t') \|e^{t'\Delta_{\rm h}} \uh_0\|^2_{\LVLH}  \\
&&\quad\qquad{}+ g^4(t')\Bigl( \int_0^{t'}
\|\uh(t'')\|_{\LVLH}\|\om^{\rm h}(t'')\|_{L^\infty_{\rm v} (L^2_{\rm
h})} dt''\Bigr)\biggr) g^2(t') \exp\Bigl
(2\int_0^{t'}g^2(t'')dt''\Bigr) dt'\,. \eeno }
\end{col}

Now let us turn to the proof of Theorem \ref{estimfondNS2.5}.

\begin{proof}[Proof of Theorem \ref{estimfondNS2.5}]
Following Wiegner's method in \cite{Wiegner}, the first step
consists in the application of this corollary to get a very rough
decay estimate on~$ \|\uh(t)\|_{\LVLH}$. Let us choose the
function~Ê$g$ such that
$$
2g^2(t) = \frac 3 T \Bigl( e+\frac t T\Bigr)^{-1}\log^{-1}
\Big(e+\frac t T\Bigr)\quad\hbox{which gives}\quad \exp\Bigl
(2\int_0^t g^2(t')dt'\Bigr) = \log^3\Bigl(e+\frac t T\Bigr)\,\cdotp
$$
Here the positive real number~$T$ is a scaling parameter which will
be chosen later on. The first inequality of
Corollary\refer{Wiegnersingperturbcorol} together with
Estimate\refeq{1.4} with~$p=\infty$ gives \beno
&&\|\uh(t)\|_{L^\infty_{\rm v} (L^2_{\rm h})}^2\log^3\Bigl(e+\frac t T\Bigr)\leq  \|\uh_0\|_{L^\infty_{\rm v} (L^2_{\rm h})}^2\\
&&\qquad\quad\qquad{}+C \int_0^t    \Bigl( e+\frac {t'} T\Bigr)^{-1} \log^2\Bigl(e+\frac {t'} T\Bigr) \|e^{t'\Delta_{\rm h}} \uh_0\|^2_{\LVLH}  \frac{dt'}T\\
&&\qquad\quad\qquad\qquad\quad\qquad{} + C \|\uh_0\|_{L^\infty_{\rm
v} (L^2_{\rm h})}^4 \int_0^t  \Bigl( e+\frac {t'} T\Bigr)^{-3}
\Bigl(\frac {t'} {T}\Bigr)^2\frac {dt'} T\,\cdotp \eeno For any
positive~$\d$ less than~$1$, we have
$$
 \|e^{t\Delta_{\rm h}} \uh_0\|^2_{\LVLH}
\leq
\frac 1 {t^{\d}} \sup_{t>0} {t^\d} \|e^{t\Delta_{\rm h}} \uh_0\|^2_{\LVLH} .
$$
Let us observe that \beq \label{definBesovelem} \sup_{t>0} {t^\d}
\|e^{t\Delta_{\rm h}} \uh_0\|^2_{\LVLH}=\sup_{z\in\Rv}  \sup_{t>0}
{t^\d} \|e^{t\Delta_{\rm h}} \uh_0(\cdot,z)\|_{L^2(\Rh)}^2 =
\|\uh_0\|_{L^\infty_{\rm v} (\dot B^{-\d}_{2,\infty}(\R_{\rm
h}^2))}. \eeq For~Ê$\d$ positive and less than~$1$,  the function~$
r\mapsto (e+r)^{-1} r^{-\d} \log^2(e+r)$ is integrable on~$\R^+$.
Thus we get that \beno
\|\uh(t)\|^2_{\LVLH} &  \leq &     C_\d(u_0,T) \log^{-2}\Bigl(e+\frac t T\Bigr) \with\\
\nonumber C_\d(u_0,T) & \eqdefa & C  \|\uh_0\|^2_{\LVLH}
\bigl(1+\|\uh_0\|^2_{\LVLH}\bigr)+  \frac {C_\d} {T^{\d}}
\|\uh_0\|_{L^\infty_v (\dot B^{-\d}_{2,\infty}(\Rh))}^2. \eeno Thus
we have \beq \label{firstdecaylog}
\begin{split}
&\quad \qquad T\geq T_\d(\uh_0)\Longrightarrow
\|\uh(t)\|^2_{\LVLH}  \leq      C_\d(\uh_0) \log^{-2}\Bigl(e+\frac t T\Bigr) \with \\
&
{} C(\uh_0)\eqdefa  \|\uh_0\|^2_{\LVLH}
\bigl(1+\|\uh_0\|^2_{\LVLH}\bigr)\andf  T_\d(\uh_0)\eqdefa
C_\d^{\f1\d} \biggl( \frac {\|\uh_0\|_{L^\infty_{\rm v}(\dot
B^{-\d}_{2,\infty}(\R_{\rm h}^2))}}
{\|\uh_0\|_{\LVLH}}\biggr)^{\frac 2 \d}\,.
\end{split}
\eeq

Now let us apply the second inequality of Corollary\refer{Wiegnersingperturbcorol} with the function~$g$ defined by
$$
2g^2(t) = \frac 1 T \Bigl(1+\frac \d 2\Bigr)  \Bigl(e+\frac t T\Bigr)^{-1} \quad\hbox{which gives}\quad
 \exp\Bigl (2\int_0^t g^2(t')dt'\Bigr) =   e^{-1}\Bigl(e+\frac t T\Bigr)^{1+\frac \d 2}\cdotp
 $$
 If we define
$$
\Om_\d(t)\eqdefa \sup_{t'\leq t}  \Bigl( e+\frac {t'}
T\Bigr)^{1+\frac \d 2} \|\om^{\rm h}(t')\|^2_{\LVLH} \,,
$$
this gives
$$
\longformule{ \Om_\d(t) \leq  e\|\na_{\rm h}\uh_0\|^2_{\LVLH} +\frac
C T\int_0^t \Bigl(e+\frac{t'}T\Bigr)^{-1+\frac \d 2}
\bigl\|e^{t'\Delta_{\rm h}} \uh_0\bigr\|^2_{\LVLH} \frac {dt'} T } {
{}+C \int_0^t \Bigl (e+\frac {t'} T\Bigr)^{-2+\frac \d2}
\Bigl(\int_0^{t'}
 \|\uh (t'')\|_{L^\infty_{\rm v} (L^2_{\rm h})}  \|\om(t'')\|_{\LVLH}  \,\frac{dt''} T\Bigr)^2 \frac {dt'} T\,\cdotp
}
$$
Using\refeq{definBesovelem} we get   that \beno \int_0^t
\Bigl(e+\frac{t'}T\Bigr)^{-1+\frac \d 2} \bigl\|e^{t'\Delta_{\rm h}}
\uh_0\bigr\|^2_{\LVLH} \frac {dt'} T & = &   \int_0^t
\Bigl(e+\frac{t'}T\Bigr)^{-1+\frac \d 2}\Bigl(\frac{t'}T\Bigr)^{-\d}
(\frac{t'}T\Bigr)^{\d} \bigl\|e^{t'\Delta_{\rm h}}
\uh_0\bigr\|^2_{\LVLH}
\frac {dt'} T\\
& \leq  & C_\d  \frac 1 {T^\d} \| \uh_0\bigr\|^2_{L^\infty_{\rm
v}(\dot B^{-\d}_{2,\infty}(\R_{\rm h}^2)) }. \eeno
 Then Estimate\refeq{firstdecaylog} and the definition of~Ê$\Om_\d$  implies that, if~$T\geq T_\d(\uh_0)$,
$$
\longformule{
 \Om_\d(t) \leq  e\|\na_{\rm h}\uh_0\|^2_{\LVLH} +C_\d
\frac 1 {T^{1+\d}} \| \uh_0\bigr\|^2_{L^\infty_{\rm v}(\dot
B^{-\d}_{2,\infty}(\R_{\rm h}^2))}
 }
 {
 {} +C(\uh_0) \int_0^t \Bigl
(e+\frac {t'} T\Bigr)^{-2+\frac \d2} \Bigl(\int_0^{t'} \Bigl(
e+\frac {t'' } T\Bigr)^{-\frac 12 -\frac \d 4}  \log^{-1} \Bigl(
e+\frac {t'' } T\Bigr)  \,\frac{dt''} T\Bigr)^2 \Om_\d(t') \frac
{dt'} T\,\cdotp
}
$$
Since~$\d$ is less than~$1$, Cauchy Schwarz
inequality with the measure $\ds \Bigl( e+\frac {t'' }
T\Bigr)^{-\frac 12 -\frac \d 4} \frac {dt''}  T $ gives,
 \beno \cI(t') &
\eqdefa &  \biggl(\int_0^{t'}
\Bigl( e+\frac {t'' } T\Bigr)^{-\frac 12 -\frac \d 4}  \log^{-1} \Bigl( e+\frac {t'' } T\Bigr)  \,\frac{dt''} T\biggr)^2  \\
&\leq &  \biggl(\int_0^{t'}
\Bigl( e+\frac {t'' } T\Bigr)^{-\frac 12 -\frac \d 4}    \,\frac{dt''} T\biggr) \biggl(\int_0^{t'}
\Bigl( e+\frac {t'' } T\Bigr)^{-\frac 12 -\frac \d 4}  \log^{-2} \Bigl( e+\frac {t'' } T\Bigr)  \,\frac{dt''} T\biggr)\\
&\leq  & C_\d \Bigl( e+\frac {t' } T\Bigr)^{1- \frac \d 2} \log^{-2}
\Bigl( e+\frac {t'} T\Bigr) \cdotp \eeno Thus we obtain
$$
\longformule{ \Om_\d(t) \leq  e\|\na_{\rm h}\uh_0\|^2_{\LVLH} +C_\d
\frac 1 {T^{1+\d}} \| \uh_0\bigr\|^2_{L^\infty_{\rm v}(\dot
B^{-\d}_{2,\infty}(\R_{\rm h}^2)) } } { +C(\uh_0) \int_0^t
\Om_\d(t') \Bigl( e+\frac {t' } T\Bigr)^{- 1} \log^{-2} \Bigl(
e+\frac {t'} T\Bigr) \frac {dt'} T\,\cdotp }
$$
Gronwall lemma implies that
$$
\Om_\d(t) \leq \Bigl( e\|\na_{\rm h}\uh_0\|^2_{\LVLH} +C_\d  \frac 1
{T^{1+\d}} \| \uh_0\bigr\|^2_{L^\infty_{\rm v}(\dot
B^{-\d}_{2,\infty}(\R_{\rm h}^2))}\Bigr)
\exp\bigl(C(\uh_0)\bigr)\,\cdotp
$$
By integration this gives, if~$T\geq T_\d(\uh_0)$, \beno
\int_0^{\infty} \|\nabla_{\rm h} \uh (t)\|_{L^\infty_{\rm
v}(L^2_{\rm h})}^2 dt&\leq&
C\int_0^\infty\Om_\d(t')\Bigl(e+\frac{t}T\Bigr)^{-1-\f{\d}2}\,dt\\
&\leq&  C\Bigl( \|\na_{\rm h}\uh_0\|^2_{\LVLH} T + C_\d \frac 1
{T^{\d}} \| \uh_0\bigr\|^2_{L^\infty_{\rm v}(\dot
B^{-\d}_{2,\infty}(\R_{\rm h}^2)) } \Bigr)\exp
\bigl(C_\d(\uh_0)\bigr)\,. \eeno Selecting ~$T=T_\d(\uh_0)$ in the
above inequality concludes the proof of
Theorem\refer{estimfondNS2.5}.
\end{proof}

\setcounter{equation}{0}
\section{The global wellposedness of (NS3D) for slowly varying initial data}

This section follows essentially the lines of\ccite{cg3} up to the
fact that the external force due to the error term is simpler  and
that we deal with less regular solutions.  Let us recall the
procedure. We consider \ben \label{definuapp} u^\e_{\rm app} (t,x) =
( \uh(t,x_h,\e x_3) , 0 )\andf
  p^\e_{\rm app}(t,x)  =   p^{{\rm h}}  (t,x_h,\e x_3).
  \een
 where~$(\uh, p^{\rm h})$ is the solution of~(NS2.5D) with initial data~$\uh_0$.  Let us search the solution of~(NS3D) as
$$
u^{\e} =  u^\e_{\rm app} +  R^{\e} .
$$
 Classical computations leads to
\begin{equation} \label{erroruapp}
 \left\{\begin{array}{l}
\displaystyle \partial_t R^\e + R^\e \cdot \nabla R^\e -\Delta
R^\e+u^\e_{\rm app} \cdot \nabla R^\e+R^\e \cdot \nabla u^\e_{\rm
app} =
 F^\e-\nabla q^\e\with\\
 \displaystyle \dive R^\e=0 \andf F^\e \eqdefa \bigl(0, \e \partial_z  p^{\rm h}\bigr)  (t,x_h,\e
 x_3),\\
\displaystyle R^\e|_{t=0} =0\,.
\end{array}\right.
\end{equation}

The first step of the proof is to prove a global existence lemma for
a perturbed Navier-Stokes system with small external force. Mover
precisely, we have the following lemma.
\begin{lem}
\label{perturaNSanisolight}
{\sl
Let us consider a divergence free vector field~$v$  such that
$$
\cN_v \eqdefa \int_0^\infty N(v(t))dt \with N(w) \eqdefa
\|w\|_{L^\infty(\R^3)}^2 +\|\nabla _h w\|_{\LVLH}^2 +\| \partial_3
w\|_{L^2_{\rm v}(\dot H^{\frac 12}(\Rh))}^2
$$
is finite and an external force $F$ in~$L^2(\R^+;\dot
H^{-\frac12}(\R^3))$. We consider the system
$$
{\rm (NS)}_v\left \{
\begin{array}{c}\partial_t w + w \cdot \nabla w -\Delta w+v \cdot \nabla w+w \cdot \nabla v = -\nabla p+F ,
\qquad (t,x)\in\R^+\times\R^3, \\
 \dive w =0  \andf w_{|t=0} = 0.
\end{array}
\right.
$$
Then a positive constant~$C_0$ exists such that if
$$
\|F\|_{L^2(\R^+;\dot H^{-\frac12}(\R^3))}^2 \leq \frac 1{C_0} \exp \bigl( C_0\cN_{v}\bigr)\,,
$$
 the system~${\rm (NS)}_v $ has a unique global solution $w$ in the
space
$$
L^\infty(\R^+; \dot H^{\frac 12} (\R^3)) \cap L^2(\R^+;\dot H^{\frac
32}(\R^3))\,.
$$
}
\end{lem}
\begin{proof}
The fact that the system~${\rm (NS)}_v$ is locally wellposed follows
from a classical  Fujita and Kato theory (\cite{fujitakato}). In
order to prove global existence part of Lemma
\ref{perturaNSanisolight}, it suffices to control
$$
\|w(t)\|_{\dot  H^{\frac 12} (\R^3)}^2+\int_0^t \|\nabla w(t')\|_{\dot \dot H^{\frac 12} (\R^3)}^2 dt'\,.
$$
Performing a~$\dot H^{\frac 12}$ energy estimate for $(NS)_v,$ we
get \beq \label{perturaNSanisolightdemoeq1}
\begin{split}
\frac 12 \frac d {dt} \|w(t)\|_{\dot  H^{\frac 12} (\R^3)}^2&
+\|\nabla w(t)\|_{\dot  H^{\frac 12} (\R^3)}^2
= -(w\cdot \nabla w|w)_{\dot  H^{\frac 12} (\R^3)} \\
&-(v\cdot \nabla w|w)_{\dot  H^{\frac 12} (\R^3)} -(w\cdot \nabla
v|w)_{\dot  H^{\frac 12} (\R^3)} -(F|w)_{\dot  H^{\frac 12} (\R^3)}
\end{split}
\eeq Law of product in Sobolev spaces implies that~$\|w\cdot \nabla
w\|_{\dot  H^{-\frac 12} (\R^3)}\leq \|w\|_{\dot  H^{\frac 12}
(\R^3)}\| \nabla w\|_{\dot  H^{\frac 12} (\R^3)}$. This gives \beq
\label{perturaNSanisolightdemoeq2} \bigl|(w\cdot \nabla w|w)_{\dot
H^{\frac 12} (\R^3)} \bigr|\leq C \|w\|_{\dot  H^{\frac 12}
(\R^3)}\| \nabla w\|_{\dot  H^{\frac 12} (\R^3)} ^2. \eeq Using that
\beq \label{perturaNSanisolightdemoeq3} (a|b)_{\dot  H^{\frac 12}
(\R^3)} \leq \|a\|_{L^2}\| \nabla b\|_{L^2}\,, \eeq we infer that
$$
 \bigl|(v\cdot \nabla w|w)_{\dot  H^{\frac 12} (\R^3)}\bigr| \leq  \|v\cdot \nabla w\|_{L^2(\R^3)} \|\nabla w\|_{L^2(\R^3)} \leq \|v\|_{L^\infty(\R^3)} \|\nabla w\|_{L^2(\R^3)}^2.
 $$
 Interpolation inequality between Sobolev spaces ensures
$$
\bigl|(v\cdot \nabla w|w)_{\dot  H^{\frac 12} (\R^3)} \bigr|\leq C
\|v\|_{L^\infty(\R^3)} \|w\|_{\dot  H^{\frac 12} (\R^3)}\| \nabla w\|_{\dot  H^{\frac 12} (\R^3)} .
$$
Then convexity inequality implies that
 \beq
\label{perturaNSanisolightdemoeq4} \bigl|(v\cdot \nabla w|w)_{\dot
H^{\frac 12} (\R^3)} \bigr|\leq \frac 1 {10}  \| \nabla w\|_{\dot
H^{\frac 12} (\R^3)} ^2+C \|v\|_{L^\infty(\R^3)}^2 \|w\|_{\dot
H^{\frac 12} (\R^3)}^2. \eeq The term~$(w\cdot \nabla v|w)_{\dot
H^{\frac 12} (\R^3)}$ is a little bit more delicate. In order to
treat it, we must take into account  some anisotropy.
Using\refeq{perturaNSanisolightdemoeq3}, we get
$$
\bigl| (w^{\rm h} \cdot \nabla_{\rm h}  v|w)_{\dot  H^{\frac 12} (\R^3)}\bigr| \leq  \|w^{\rm h}\cdot \nabla _{\rm h}v\|_{L^2(\R^3)}
\|\nabla w\|_{L^2(\R^3)}.
$$
Interpolation theory implies that
 \beq
\label{perturaNSanisolightdemoeq5}
\|a\|_{L^\infty(\Rh)} \lesssim \|a\|_{\dot H^{\frac 12}(\Rh)} ^{\frac 12}
 \|\nabla_{\rm h}a\|_{\dot H^{\frac 12}(\Rh)} ^{\frac 12} .
\eeq
We deduce that for any~$x_3$ in~$\Rv$, we have
$$
\|w^{\rm h}(\cdot, x_3) \cdot \nabla _{\rm h}v(\cdot
,x_3)\|_{L^2(\Rh)}\lesssim \|w(\cdot, x_3)\|_{\dot H^{\frac12}
(\Rh)}^{\frac12}\|\nabla w(\cdot, x_3)\|_{\dot H^{\frac12}
(\Rh)}^{\frac12} \|\nabla_{\rm h} v\|_{\LVLH}.
$$
As obviously~$\|a\|_{L^2(\Rv; \dot H^{\frac12} (\Rh))}$ is less than or equal to~$\|a\|_{\dot H^{\frac12} (\R^3)}$, we deduce that
$$
\bigl| (w^{\rm h} \cdot \nabla_{\rm h}  v|w)_{\dot  H^{\frac 12} (\R^3)}\bigr| \lesssim
\|w\|_{\dot H^{\frac12} (\R^3)}^{\frac 12}\|\nabla w\|_{\dot H^{\frac12} (\R^3)}^{\frac 12}
 \|\nabla_{\rm h} v\|_{\LVLH} \|\nabla w\|_{L^2(\R^3)}.
$$
Interpolation inequality between Sobolev spaces and convexity
inequality yields \beq \label{perturaNSanisolightdemoeq6} \bigl|
(w^{\rm h} \cdot \nabla_{\rm h}  v|w)_{\dot  H^{\frac 12}
(\R^3)}\bigr| \leq
 \frac 1 {10}  \| \nabla w\|_{\dot  H^{\frac 12} (\R^3)} ^2+C
\|\nabla_{\rm h} v\|_{\LVLH}^2 \|w\|_{\dot  H^{\frac 12} (\R^3)}^2.
\eeq Now let us examine the term which involves vertical derivative.
Using again\refeq{perturaNSanisolightdemoeq3}, we are reduced to
estimate~$\|w^3\partial_3v\|_{L^2(\R^3)}$. Using law of product of
Sobolev spaces in~$\Rh$, we get, for any~$x_3$ in~$\Rv$,
$$
\|w^3(\cdot,x_3)\partial_{3} v(\cdot, x_3) \|_{L^2_{\rm h} } \leq C
\|w(\cdot ,x_3)\|_{\dot H^{\frac 12}(\Rh)}  \|\partial_{3} (\cdot
,x_3)\|_{\dot H^{\frac 12}(\Rh)}  .
$$
Observing that if~$\cH$ is a Hilbert space and $a$ is regular
function from $\Rv$into~$\cH$, we can write
$$
\|a(x_3)\|_{\cH}^2 =2\int_{-\infty} ^{x_3}\bigl( \partial_{y_3}
a(y_3)\big| a(y_3)\bigr)_{\cH} dy_3 .
$$
Then using Cauchy-Schwarz  inequality, we get \beq
\label{perturaNSanisolightdemoeq6b} \|a\|^2_{L^\infty(\R;\cH)} \leq
2 \|a\|_{L^2(\R;\cH)} \| \partial_{3} a\|_{L^2(\R;\cH)}. \eeq We
infer that
$$
\|w\|_{L^\infty_{\rm v}(\dot H^{\frac 12}(\Rh))}  \leq \sqrt 2\,
\|w\|_{L^2(\R_{\rm v};\dot H^{\frac 12}(\Rh))} ^{\frac 12}\| \p_{3}
w\|_{L^2(\R_{\rm v};\dot H^{\frac 12}(\Rh))}^{\frac 12} \leq \sqrt
2\, \|w\|_{\dot H^{\frac 12}(\R^3))}^{\frac 12} \| \p_{3} w\|_{\dot
H^{\frac 12}(\R^3)}^{\frac 12}.
$$
Thus
$$
\|w^3\partial_{3} v \|_{L^2(\R^3)} \leq C
 \|w\|_{\dot H^{\frac 12}(\R^3))}^{\frac 12} \| \p_{3} w\|_{\dot H^{\frac 12}(\R^3)}^{\frac 12}
\|\partial_{3} v\|_{L^2_{\rm v}(\dot  H^{\frac 12} (\Rh))} .
$$
After interpolation and convexity inequality, we infer that \beq
\label{perturaNSanisolightdemoeq6as} \begin{split} \bigl| (w^{3}
\cdot \p_3 v|w)_{\dot H^{\frac 12} (\R^3)}\bigr|\leq
&\|w^3\partial_{3} v \|_{L^2(\R^3)}\|\na w\|_{L^2(\R^3)}\\
\leq & \frac 1 {10} \| \nabla w\|_{\dot H^{\frac 12}(\R^3)}^2 +C
\|\partial_{3} v\|_{L^2_{\rm v}(\dot  H^{\frac 12}
(\Rh))}^2\|w\|_{\dot H^{\frac 12}(\R^3)}^2 . \end{split} \eeq As we
have
$$
(F|w)_{\dot H^{\frac 12}(\R^3)} \leq \frac 1{10}  \|\nabla
w\|^2_{\dot H^{\frac 12}(\R^3)}+ 10\|F\|^2_{\dot H^{-\frac
12}(\R^3)},
$$
we infer from\refeq{perturaNSanisolightdemoeq1},
\refeq{perturaNSanisolightdemoeq2},
\refeq{perturaNSanisolightdemoeq4},
\refeq{perturaNSanisolightdemoeq6} and
\refeq{perturaNSanisolightdemoeq6as} that
$$
\longformule { \frac 12 \frac d {dt} \|w(t)\|_{\dot  H^{\frac 12}
(\R^3)}^2 +\frac 3 5 \|\nabla w(t)\|_{\dot  H^{\frac 12} (\R^3)}^2
\leq  C\|w(t)\|_{\dot  H^{\frac 12} (\R^3)}\| \nabla w(t)\|_{\dot
H^{\frac 12} (\R^3)} ^2 } { {}+C N(v(t))  \|w(t)\|_{\dot  H^{\frac
12} (\R^3)}^2 +10 \|F\|^2_{\dot H^{\frac 12}(\R^3)}. }
$$
Defining $\ds w_\lam(t) \eqdefa \exp \Bigl( -\lam \int_0^t
N(v(t'))dt'\Bigr).$ Then for $\lam\geq C,$ we deduce that
$$
\frac 12 \frac d {dt} \|w_\lam(t)\|_{\dot  H^{\frac 12} (\R^3)}^2 +\frac 3 5 \|\nabla w_\lam(t)\|_{\dot  H^{\frac 12} (\R^3)}^2
\leq  C\|w(t)\|_{\dot  H^{\frac 12} (\R^3)}\| \nabla w_\lam(t)\|_{\dot  H^{\frac 12} (\R^3)} ^2 + 10\|F\|^2_{\dot H^{\frac 12}(\R^3)}.
$$
Thus as long as
\beq
\label{perturaNSanisolightdemoeq7}
C\|w(t)\|_{\dot  H^{\frac 12} (\R^3)} \leq \frac 1 {10}\,\virgp
\eeq
we have that
$$
 \frac d {dt} \|w_\lam(t)\|_{\dot  H^{\frac 12} (\R^3)}^2 + \|\nabla w_\lam(t)\|_{\dot  H^{\frac 12} (\R^3)}^2
\leq 20\|F\|^2_{\dot H^{\frac 12}(\R^3)}.
$$
So that whenever the Condition\refeq{perturaNSanisolightdemoeq7} is
satisfied, we have
$$
 \|w(t)\|_{\dot  H^{\frac 12} (\R^3)}^2 +\int_0^t  \|\nabla w(t')\|_{\dot  H^{\frac 12} (\R^3)}^2 dt'
\leq 20\|F\|^2_{L^2(\R^+;\dot H^{\frac 12}(\R^3))}
\exp\bigl(\lam\cN_v\bigr)\,.
$$
Hence by a very classical  continuation argument, we get that, if
$$
20\|F\|^2_{L^2(\R^+;\dot H^{\frac 12}(\R^3))}
\exp\bigl(\lam\cN_v\bigr)\leq \frac 1 {121C^2}\,\virgp
$$
then Condition\refeq{perturaNSanisolightdemoeq7} is always satisfied
and  the norms~$L^\infty(\R^+;\dot
H^{\frac 12}(\R^3))$ and~$L^2(\R^+;\dot H^{\frac 32}(\R^3))$  of the solution are controlled and the
lemma is proved.
\end{proof}

Now let us compute~$N(u^\e_{\rm app}(t))$. In view of
\eqref{definuapp} and because of the vertical scaling, we have \beq
\label{4.12} N(\uapp(t)) = \|\uh(t)\|_{L^\infty(\R^3)}^2 +
\|\nabla_{\rm h} \uh(t)\|_{\LVLH}^2 + \e \|\partial_z
\uh(t)\|_{L^2_{\rm v}(\dot H^{\frac 12}(\Rh))}^2. \eeq

In order to control~$N(u^\e_{\rm app}(t)),$ we need the following
propagation lemma.
\begin{lem}
\label{progaanisoelem} {\sl Let~$\uh$ be a regular solution
of~(NS2.5D). Then for any~$s$ between~$-1$ and~$1,$ we have
 \beno
&&\|\uh (t) \|_{L^\infty_{\rm v}(\dot H^s(\Rh))}^2\leq  \|\uh_0
\|_{L^\infty_{\rm v}(\dot H^s(\Rh))}^2 \exp \biggl (C_s  \int_0^t
\|\nabla_{\rm h} \uh(t')\|_{\LVLH}^2dt'\biggr),\\
&&\|\uh (t) \|_{L^2_{\rm v}(\dot H^s(\Rh))}^2 +\int_0^t \|\nabla_{\rm h}\uh (t') \|_{L^2_{\rm v}(\dot H^s(\Rh))}^2 dt'\\
&&\qquad\qquad\qquad{} \leq  \|\uh_0 \|_{L^2_{\rm v}(\dot
H^s(\Rh))}^2 \exp \biggl (C_s  \int_0^t  \|\nabla_{\rm h}
\uh(t')\|_{\LVLH}^2dt'\biggr),\eeno and \beno  &&\|\partial_z \uh
(t) \|_{L^2_{\rm v}(\dot H^s(\Rh))}^2 +\int_0^t \|\nabla_{\rm
h}\partial_z\uh (t') \|_{L^2_{\rm v}(\dot H^s(\Rh))}^2 dt'
\\
&&\qquad\qquad\qquad{} \leq  \|\partial_z\uh_0 \|_{L^2_{\rm v}(\dot
H^s(\Rh))}^2 \exp \biggl (C_s \int_0^t  \|\nabla_{\rm h}
\uh(t')\|_{\LVLH}^2dt'\biggr). \eeno }
\end{lem}
\begin{proof}
Let us perform a~$\dot H^s$ energy estimate in the horizontal
variables for (NS2.5D). This gives, for any~$z$ in~$\Rv$,
$$
\longformule { \frac 12 \frac d {dt}  \|\uh(t,\cdot,z)\|_{\dot
H^s(\Rh)}^2 + \|\nabla_\e\uh(t,\cdot,z)\|_{\dot H^s(\Rh)}^2  -\frac
{\e^2} 2 \partial_z^2 \|\uh(t,\cdot,z)\|_{\dot H^s(\Rh)}^2 } { {}= -
\bigl( \uh(t,\cdot,z)\cdot\nabla_{\rm h} \uh(t,\cdot,z) \big|
\uh(t,\cdot,z)\bigr)_{\dot H^s(\Rh)}. }
$$
Lemma~1.1 of\ccite{chemin10} implies that
$$
\longformule { \big|\bigl( \uh(t,\cdot,z)\cdot\nabla_{\rm h}
\uh(t,\cdot,z) \big| \uh(t,\cdot,z)\bigr)_{\dot H^s(\Rh)}\bigr| }
{{} \lesssim  \|\nabla_{\rm h}\uh(t,\cdot,z)\|_{\dot H^s(\Rh)}\|
\nabla_{\rm h} \uh(t,\cdot,z)\|_{L^2(\Rh)} \|\uh(t,\cdot,z)\|_{\dot
H^s(\Rh)}\,. }
$$
Convexity inequality implies that
$$
\longformule { \big|\bigl( \uh(t,\cdot,z)\cdot\nabla_{\rm h}
\uh(t,\cdot,z) \big| \uh(t,\cdot,z)\bigr)_{\dot H^s(\Rh)}\bigr| }
{{} \leq\frac 1 2  \|\nabla_{\rm h}\uh(t,\cdot,z)\|_{\dot
H^s(\Rh)}^2 +C \| \nabla_{\rm h} \uh(t,\cdot,z)\|_{L^2(\Rh)}^2
\|\uh(t,\cdot,z)\|_{\dot H^s(\Rh)}^2. }
$$
Defining
$$
\uh_\lam(t,\cdot,z) \eqdefa \exp\biggl(-\lam \int_0^t \|\nabla_{\rm
h} \uh(t')\|_{\LVLH} dt'\biggr) \uh(t,\cdot,z).
$$
Then for $\lam\geq C,$ we can write
$$
 \frac d {dt}  \|\uh_\lam(t,\cdot,z)\|_{\dot H^s(\Rh)}^2
+ \|\nabla_\e\uh_\lam(t,\cdot,z)\|_{\dot H^s(\Rh)}^2  -\frac {\e^2} 2 \partial_z^2 \|\uh_\lam(t,\cdot,z)\|_{\dot H^s(\Rh)}^2 \leq 0.
$$
We get the first inequality of the lemma by maximal principle, and
the second one by integration in~$z$ and in time.

In order to prove the third inequality, we take $\p_z$ to the
system~(NS2.5D)  and then perform a $\dot H^s$ energy estimate in
the horizontal variables for the resulting equation. This gives, for
any~$z$ in~$\Rv$, \ben \nonumber&&\frac12  \frac d {dt}
\|\partial_z\uh(t,\cdot,z)\|_{\dot H^s(\Rh)}^2
+ \|\nabla_\e\partial_z\uh(t,\cdot,z)\|_{\dot H^s(\Rh)}^2  -\frac {\e^2} 2 \partial_z^2 \|\partial_z\uh(t,\cdot,z)\|_{\dot H^s(\Rh)}^2 \\
\label{progaanisoelemdemoeq1}&&\qquad\qquad\qquad\qquad\qquad{}= - \bigl( \uh(t,\cdot,z)\cdot\nabla_{\rm h} \partial_z\uh(t,\cdot,z) \big|\partial_z\uh(t,\cdot,z)\bigr)_{\dot H^s(\Rh)}\\
\nonumber&&\qquad\qquad\qquad\qquad\qquad\qquad\qquad\qquad{}-\bigl(
\partial_z\uh(t,\cdot,z)\cdot\nabla_{\rm h} \uh(t,\cdot,z)
\big|\partial_z\uh(t,\cdot,z)\bigr)_{\dot H^s(\Rh)}. \een Again
Lemma~1.1 of\ccite{chemin10} implies that \beq
\label{progaanisoelemdemoeq2}
\begin{split}
&\big|\bigl( \uh(t,\cdot,z)\cdot\nabla_{\rm h} \partial_z\uh(t,\cdot,z) \big| \partial_z\uh(t,\cdot,z)\bigr)_{\dot H^s(\Rh)}\bigr|\\
&\qquad\qquad\qquad{} \lesssim  \|\nabla_{\rm
h}\uh(t,\cdot,z)\|_{L^2(\Rh)}\| \partial_z\nabla_{\rm h}
\uh(t,\cdot,z)\|_{\dot H^s(\Rh)} \|\partial_z\uh(t,\cdot,z)\|_{\dot
H^s(\Rh)}\,.
\end{split}
\eeq If~$s=0$, we use Sobolev embeddings and interpolation theory to
write
\beno
 \!\!\!\!\bigl | \bigl(
\partial_z\uh(t,\cdot,z)\cdot\nabla_{\rm h} \uh(t,\cdot,z)
\big|\partial_z\uh(t,\cdot,z)\bigr)_{L^2(\Rh)} \bigr|
&\leq & \|\nabla_{\rm h} \uh(t,\cdot,z)\|_{L^2(\Rh)}\| \partial_z\uh(t,\cdot,z)\|^2_{L^4(\Rh)}\\
& \lesssim & \|\nabla_{\rm h} \uh(t,\cdot,z)\|_{L^2(\Rh)}
\| \partial_z\uh(t,\cdot,z)\|_{L^2(\Rh)}\\
&&\qquad \qquad \qquad\quad  {}\times \|\nabla_{\rm h}
\partial_z\uh(t,\cdot,z)\|_{L^2(\Rh)}.
\eeno
If~$s$ is different
from~$0$,  laws of product in Sobolev space imply that
$$
\|\partial_z\uh(t,\cdot,z)\cdot\nabla_{\rm
h}\uh(t,\cdot,z)\|_{\dot H^{s-1}(\Rh)} \lesssim
\|\partial_z\uh(t,\cdot,z)\|_{\dot H^{s}(\Rh)} \|\nabla_{\rm h}
\uh(t,\cdot,z)\|_{L^2(\Rh)}
$$
if~$s$ belongs to~$]0,1[$. If~$s$ is in~$]-1,0[$, we have
$$
\|\partial_z\uh(t,\cdot,z)\cdot\nabla_{\rm
h}\uh(t,\cdot,z)\|_{\dot H^{s}(\Rh)} \lesssim
\|\partial_z\nabla_{\rm h}\uh(t,\cdot,z)\|_{\dot H^{s}(\Rh)}
\|\nabla_{\rm h} \uh(t,\cdot,z)\|_{L^2(\Rh)}
$$
Plugging these estimates
and\refeq{progaanisoelemdemoeq2} into\refeq{progaanisoelemdemoeq1}
and using convexity inequality, we obtain
$$
\longformule{   \frac d {dt} \|\partial_z\uh(t,\cdot,z)\|_{\dot
H^s(\Rh)}^2 + \|\nabla_\e\partial_z\uh(t,\cdot,z)\|_{\dot
H^s(\Rh)}^2  -\frac {\e^2} 2 \partial_z^2
\|\partial_z\uh(t,\cdot,z)\|_{\dot H^s(\Rh)}^2 } { {} \leq C \|
\nabla_{\rm h} \uh(t,\cdot,z)\|_{L^2(\Rh)}^2
\|\partial_z\uh(t,\cdot,z)\|_{\dot H^s(\Rh)}^2. }
$$
Arguing as in the proof of the second inequality allows to conclude
the proof of the lemma.
\end{proof}
We can deduce from this lemma the following corollary.
\begin{col}
\label{estimuappco} {\sl Let~$\uh$ be a regular solution
of~(NS2.5D). Then  under the assumptions of
Theorem\refer{slowvarsimplifie}, one has, for $A_\d(\uh_0)$
given by \eqref{estiomglobalfond},
\beq
\label{827z}
\begin{split}
\cN_{u^\e_{\rm app}} \,\lesssim \,& A_\d(\uh_0)+ U_0\exp\bigl( CA_\d(\uh_0)\bigr)\with\\
U_0 \,\eqdefa & \,\e\|\partial_z \uh_0\|_{L^2_{\rm v}(\dot H^{-\frac
12}(\Rh))}^2+ \|\uh_0\|_{L^2_{\rm v}(\dot H^{-\frac
12}(\Rh))}^{\frac12}
 \|\partial_z\uh_0\|_{L^2_{\rm v}(\dot H^{-\frac 12}(\Rh))}^{\frac12}
\\
&\qquad\qquad\qquad\qquad\qquad\qquad {}\times \|\uh_0\|_{L^2_{\rm v}(\dot
H^{\frac 12}(\Rh))}^{\frac12} \|\partial_z\uh_0\|_{L^2_{\rm v}(\dot
H^{\frac 12}(\Rh))}^{\frac12}.
\end{split}
\eeq
}
\end{col}
\begin{proof}
Lemma\refer{progaanisoelem} implies that
\beq
\label{estimuappcodemoeq1}
\int_0^\infty \|\partial_z \uh(t)\|_{L^2_{\rm v}(\dot H^{\frac 12}(\Rh))}^2 dt  \leq
\|\partial_z \uh_0\|_{L^2_{\rm v}(\dot H^{-\frac 12}(\Rh))}^2
\exp\biggl(C \int_0^\infty \|\nabla_{\rm h} \uh(t,\cdot)\|_{\LVLH}^2dt\biggr).
\eeq
As we have
$$
\|\uh(t,\cdot,z)\|^2_{L^\infty_{\rm h}} \lesssim \|\uh(t,\cdot,z)\|_{\dot H^{\frac 12}(\Rh)}
\|\nabla_{\rm h}\uh(t,\cdot,z)\|_{\dot H^{\frac 12}(\Rh)}.
$$
Using\refeq{perturaNSanisolightdemoeq6b}, we infer that
$$
\longformule{
\|\uh(t)\|^2_{L^\infty(\R^3)} \lesssim  \|\uh(t)\|_{L^2(\Rv;\dot H^{\frac 12}(\Rh))}^{\frac12}
 \|\partial_z\uh(t)\|_{L^2(\Rv;\dot H^{\frac 12}(\Rh))}^{\frac12} }
 {
{}\times \|\nabla_{\rm h}\uh(t)\|_{L^2(\Rv;\dot H^{\frac 12}(\Rh))}^{\frac12}
\|\partial_z\nabla_{\rm h}\uh(t)\|_{L^2(\Rv;\dot H^{\frac 12}(\Rh))}^{\frac12} .
}
$$
Lemma\refer{progaanisoelem} implies that
$$
\longformule{ \int_0^\infty \|\uh(t)\|^2_{L^\infty(\R^3)} dt
\lesssim   \|\uh_0\|_{L^2_{\rm v}(\dot H^{-\frac
12}(\Rh))}^{\frac12}
 \|\partial_z\uh_0\|_{L^2_{\rm v}(\dot H^{-\frac 12}(\Rh))}^{\frac12}
 }
 {
{}\times \|\uh_0\|_{L^2_{\rm v}(\dot H^{\frac 12}(\Rh))}^{\frac12}
\|\partial_z\nabla_{\rm h}\uh_0\|_{L^2_{\rm v}(\dot H^{\frac
12}(\Rh))}^{\frac12} \exp\biggl(C \int_0^\infty \|\nabla_{\rm h}
\uh(t,\cdot)\|_{\LVLH}^2dt\biggr). }
$$
Together with \eqref{4.12} and Theorem\refer{estimfondNS2.5}, this
ensures \eqref{827z}.
\end{proof}

Finally let us present the proof of Theorem\refer{slowvarsimplifie}.

\begin{proof}[Proof of Theorem\refer{slowvarsimplifie}] By virtue of Lemma \ref{perturaNSanisolight} and Corollary
\ref{estimuappco}, in order to conclude the proof of the Theorem, it
amounts to prove that \beq \label{ineqconclude}
\|F_\e\|_{L^2(\R^+;\dot H^{-\frac 12}(\R^3))} \lesssim C(\uh_0)\,
\e^{\frac 12}. \eeq Toward this, we get, by first applying the
operator~$\dive_{\rm h}$ to the equation~(NS2.5D) and then taking
$\p_z$ to the resulting equation, that
$$
\partial_z p^{{\rm h}} (t,\cdot, ,z) = 2\sum_{1\leq j,k\leq 2} (-\D_{\rm h})^{-1} \partial_j\partial_k\bigl(u^{{\rm h},j}(t,\cdot, z)\partial_z u^{{\rm h},k}(t,\cdot,z)\bigr).
$$
Law of product for Sobolev spaces for the horizontal variables
implies that
$$
\bigl\|u^{{\rm h},j}(t,\cdot, z)\partial_z u^{{\rm h},k}(t,\cdot,z)\bigr\|_{\dot H^{-\frac 12}(\Rh)}
\lesssim \|u^{{\rm h}}(t,\cdot, z)\|_{L^2_{\rm h}} \|\partial_z u^{{\rm h}}(t,\cdot,z)\bigr\|_{\dot H^{\frac 12}(\Rh)}\,.
$$
As~$(-\D_{\rm h})^{-1} \partial_j\partial_k$ is a bounded Fourier multiplier, we get that
$$
\|\partial_z p^{{\rm h}}(t)\|_{L^2_{\rm v}(\dot H^{-\frac 12}(\Rh))}
\lesssim
 \|u^{{\rm h}}(t)\|_{\LVLH} \|\partial_z u^{{\rm h}}(t)\bigr\|_{L^2_{\rm v}(\dot H^{\frac 12}(\Rh))}.
 $$
 Changing variable~Ê$z$ into~$\e x_3$ gives
$$
\|F_\e(t)\|_{L^2_{\rm v}(\dot H^{-\frac 12}(\Rh))} \lesssim
\e^{\frac 12} \|\partial_z p^{{\rm h}}(t)\|_{L^2_{\rm v}(\dot
H^{-\frac 12}(\Rh))} \lesssim \e^{\frac 12}
 \|u^{{\rm h}}(t)\|_{\LVLH} \|\partial_z u^{{\rm h}}(t)\bigr\|_{L^2_{\rm v}(\dot H^{\frac 12}(\Rh))}.
$$
Hence we obtain \beno \|F_\e\|_{L^2(\R^+;\dot H^{-\frac 12}(\R^3))}
&\lesssim& \|F_\e\|_{L^2(\R^+;L^2_{\rm v}(\dot H^{-\frac
12}(\Rh)))}\\
 &\lesssim& \e^{\frac 12}
 \|u^{{\rm h}}\|_{L^\infty_t(\LVLH)} \|\partial_z u^{{\rm h}}\bigr\|_{L^2(\R^+;L^2_{\rm v}(\dot H^{\frac
 12}(\Rh)))}.
 \eeno
Lemma\refer{progaanisoelem} allows to conclude the proof of
\eqref{ineqconclude} and thus of  Theorem\refer{slowvarsimplifie}.
\end{proof}

\medskip

\noindent {\bf Acknowledgments.} Part of this work was done when
Jean-Yves Chemin was visiting Morningside Center of the Academy of
Mathematics and Systems Sciences, CAS. He appreciates the
hospitality and the financial support from MCM and  National Center
for Mathematics and Interdisciplinary Sciences. P. Zhang is
partially supported by NSF of China under 11371347, the fellowship
from Chinese Academy of Sciences and innovation grant from National
Center for Mathematics and Interdisciplinary Sciences.
\medskip



\begin{thebibliography}{9999}

 \bibitem{adt}
P.~Auscher, S.~Dubois, and P.~Tchamitchian, On the stability of global solutions to {N}avier-Stokes equations in  the space,  {\it Journal de Math\'ematiques Pures et Appliqu\'ees}, {\bf 83}, 2004, pages 673-697.



\bibitem{bcg2}
 H. Bahouri, J.-Y. Chemin and I. Gallagher,
 Stability by  rescaled weak convergence for the Navier-Stokes
 equations,  {\it arXiv:1310.0256}.


\bibitem{bg}   H. Bahouri and I. Gallagher, On the stability in   weak topology  of the set of global solutions to the Navier-Stokes equations,  {\it Archive for Rational Mechanics
and Analysis}, {\bf 209}, 2013, pages 569-629.


\bibitem{brandolese}
L. Brandolese, Asymptotic behavior of the energy and pointwise
estimates for solutions, to the Navier-Stokes equations, {\it
Revista  Mathematica Iberoamericana}, {\bf  20}, 2004, page
223-256.


\bibitem{chemin10} J.-Y. Chemin,
Remarques sur l'existence pour le syst\`eme de Navier-Stokes incompressible,
{\it  SIAM Journal of Mathematical Analysis},
{\bf 23}, 1992, pages~20--28.

 \bibitem{cg3} J.-Y. Chemin and I. Gallagher, Large, global solutions to the
Navier-Stokes equations, slowly varying in one direction,  {\it Transactions of the American Mathematical Society}, {\bf 362}, 2010,  pages 2859-2873.


 \bibitem{cgm}
J.-Y. Chemin,  I. Gallagher and C. Mullaert, The role of spectral
anisotropy in the resolution of the  three-dimensional Navier-Stokes
equations, ``Studies in Phase Space Analysis with Applications to
PDEs", {\it Progress in Nonlinear Differential Equations and Their
Applications} {\bf 84}, Birkhauser, pages 53-79, 2013.


 \bibitem{cgz}
J.-Y. Chemin,  I. Gallagher and P. Zhang, Sums of large global
solutions  to the incompressible  Navier-Stokes equations,  {\it
Journal f\"ur die reine und angewandte Mathematik}, {\bf 681}, 2013,
pages 65-82.

\bibitem{fujitakato}
H. Fujita and T. Kato, On the Navier-Stokes initial value problem I,
{\em Archive for Rational Mechanic Analysis}, {\bf 16}, 1964, pages
269--315.

 \bibitem{gipcras}
I. Gallagher,  D. Iftimie and F. Planchon,
 Non-explosion en temps grand et stabilit\'e de solutions globales des \'equations de Navier-Stokes,
 {\it Comptes  Rendus Math\'ematiques de l'Acad\'emie des Sciences de  Paris}, {\bf 334} (4),  2002, pages 289-292.

 \bibitem{gip}
I. Gallagher,  D. Iftimie and F. Planchon,  Asymptotics
   and stability for global solutions to the
   Navier--Stokes equations, {\it  Annales de l'Institut Fourier},  {\bf 53}, 2003, pages 1387-1424.

\bibitem{kochtataru} H. Koch and D. Tataru, Well--posedness for the
Navier--Stokes equations, {\it Advances in
Mathematics},  {\bf 157}, 2001, pages~22-35.



\bibitem{Schonbek} M. Schonbek,
Lower Bounds of Rates of Decay for Solutions to the
Navier-Stokes equations,
{\it Journal of the  American  Mathematical  Society}, {\bf 4}, 1991, pages  423-449.

\bibitem{Schonbek2}
M. Schonbek and T.  Schonbek,
On the boundedness and decay
of moments of solutions of the NavierÐStokes equations.
{\it Advances in  Differential Equations}, {\bf  5}, 2000, pages  861-898.


\bibitem{Wiegner} M. Wiegner, Decay results for weak solutions to the Navier-Stokes equations on~${\R}^n$,
{\it Journal of the  London Mathematical  Society}, {\bf 35}, 1987, pages 303--313.

\end{thebibliography}
\end{document}